\documentclass[12pt,a4paper]{amsart}
 \usepackage{latexsym,amssymb}

\renewcommand{\ss}{\scriptscriptstyle} 
\renewcommand{\d}{\displaystyle} 
\renewcommand{\ni}{\noindent} 
\renewcommand{\epsilon}{\varepsilon} 

\renewcommand{\le}{\leqslant}
\renewcommand{\leq}{\leqslant}

\renewcommand{\geq}{\geqslant}
\newcommand{\V}{{\boldsymbol{V}}} 
 
\renewcommand{\L}{{\mathcal{L}}} 
\newcommand{\M}{{\mathcal{M}}} 
 
\newcommand{\1}{{\boldsymbol{1}}} 
\newcommand{\F}{{\mathcal{F}}} 
\newcommand{\Q}{{\mathcal{Q}}} 
\renewcommand{\S}{{T}} 
\renewcommand{\phi}{{\varphi}}

\newcommand{\cZ}{{\mathcal{Z}}}
\newcommand{\N}{{\mathbb{N}}} 
\newcommand{\C}{{\mathbb{C}}} 
 
\newcommand{\scirc}{{\scriptstyle\circ}} 
\newcommand{\probability}{{\mathbb P}}
\newcommand{\expectation}{{\mathbb E}}
 \font\bigmath=cmmi10 scaled \magstep2 
	\newcommand{\bigchi}{\hbox{\bigmath \char31}} 
	\newcommand{\ldelta}[1]{{\L}_\Delta^{#1}} 
	\newcommand{\var}{{\mathrm{var}}}
\newcommand{\Exp}{{\expectation}}

\newcommand{\pdelta}{\!{\scriptscriptstyle \Delta}}

\newtheorem{theorem}{Theorem} 
 
\newtheorem{lemma}[theorem]{Lemma} 
\newtheorem{proposition}[theorem]{Proposition}

\newtheorem{alemma}{Lemma}[section]

\title[Poisson processes for subshifts of finite type]
{Poisson processes for subsystems of finite type in symbolic dynamics}
\author{Jean-Ren\'e Chazottes, Zaqueu Coelho and Pierre Collet} %
\address{Centre de Physique Th\'eorique\\
CNRS UMR 7644\\
Ecole Polytechnique\\
F-91128 Palaiseau Cedex\\
FRANCE} %
\email{jean-rene.chazottes@cpht.polytechnique.fr} %
\address{Department of Mathematics \\ University of York\\ Heslington \\
York YO10 5DD\\ UK} %
\email{zc3@york.ac.uk} %
\address{Centre de Physique Th\'eorique\\
CNRS UMR 7644\\
Ecole Polytechnique\\
F-91128 Palaiseau Cedex\\
FRANCE} %
\email{pierre.collet@cpht.polytechnique.fr} %
\date{\today
} %

\begin{document}

\begin{abstract}
Let $\Delta\subsetneq\V$ be a proper subset of the vertices $\V$ of 
the defining graph of an irreducible and aperiodic shift of finite type 
$(\Sigma_{A}^{+},\S)$.  Let $\Sigma_{\Delta}$ be the subshift of 
allowable paths in the graph of $\Sigma_{A}^{+}$ which only passes 
through the vertices of $\Delta$.  For a random point $x$ chosen with 
respect to an equilibrium state $\mu$ of a H\"older potential $\phi$ 
on $\Sigma_{A}^{+}$, let $\tau_{n}$ be the point process defined as 
the sum of Dirac point masses at the times $k>0$, suitably rescaled,
for which the first $n$-symbols of $\S^k x$ belong to $\Delta$. 
We prove that this point process converges in law to a marked Poisson 
point process of constant parameter measure.  The scale is related to 
the pressure of the restriction of $\phi$ to $\Sigma_{\Delta}$ and the 
parameters of the limit law are explicitly computed.
\end{abstract}

\maketitle

%
%

\section*{Introduction}
\ni The study of  limit laws for the (rescaled) random times of
occurrence of asymptotically rare events has motivated the
consideration of dynamically defined hitting time point processes (see
definition in Section~\ref{sec:aper} and the expository
notes~\cite{Coe}).  Special attention has been given to the case where
one considers, for an ergodic dynamical system on a compact metric
space, the first hitting time of shrinking neighbourhoods of a generic point.
In fact,  given any aperiodic ergodic dynamical system, 
one can get any limit law within a large class of laws \cite{lacroix}
by using a suitable family of shrinking neighbourhoods. 
In contrast with this abstract result, when one considers cylinder
sets about a generic point of a system mixing ``sufficiently well"
its partition, one expects and gets an exponential limit law for the first hitting time,
and a Poisson law for the hitting time process; see for instance
the papers~\cite{abadi,CG1,CG2,DGS,denkerkan,dolgo,Hay, HV,Hir1,Hir2,HSV,Pit}.
Here, as in~\cite{CC}, we consider another case where the intersection
of the shrinking neighbourhoods contains a non-trivial invariant set
and we show that a marked Poisson point process appears as the
asymptotic limit law.

This paper is motivated by the main result in~\cite{CC}, but phrased 
in the context of symbolic dynamical systems.  The problem reads as 
follows: given two points $x$ and $y$ in an aperiodic shift of finite 
type (randomly chosen according to the equilibrium state of some H\"older potential) 
consider the times when their orbits under the shift get 
$\epsilon$-close with respect to the usual distance
(take the distance between two points as $\sum_i |x_i-y_i|\, \rho^i$ for some fixed $0<\rho<1$, say). 
This defines a point process depending on $\epsilon$ and the problem is to prove 
whether this point process rescaled by $\epsilon^{-1}$ converges in 
law when $\epsilon$ tends to zero.  Tackling this problem with the 
technique developed in~\cite{CC} requires some additional ideas from 
spectral properties of a perturbed transfer operator studied 
in~\cite{CMS}, and we realised that the above problem could be 
obtained as an application of a much more general result after the 
study of random times of asymptotic approach to a subsystem of finite 
type, and this is the main subject of this paper.

Let $\Sigma_{\Delta}$ be a proper subshift of finite type of a 
one-sided irreducible and aperiodic shift of finite type $\Sigma_{A}^{+}$ and consider 
a generic point $x\in\Sigma_{A}^{+}$ with respect to the equilibrium 
state $\mu$ of some H\"older potential $\phi$ defined on 
$\Sigma_{A}^{+}$.  Define a point process on $[0,\infty)$ by summing 
Dirac masses at the times $t>0$ for which the orbit of $x$ under the 
shift on $\Sigma_{A}^{+}$ is $\epsilon$-close to $\Sigma_{\Delta}$.  
Using a higher-block representation of $\Sigma_{A}^{+}$ (see for 
instance~\cite{LM}), we may assume the subshift $\Sigma_{\Delta}$ is 
constructed by choosing a proper subset of vertices 
$\Delta\subsetneq\V$ of the defining graph of $\Sigma_{A}^{+}$, and 
defining $\Sigma_{\Delta}$ as the subshift of allowable paths in the 
graph of $\Sigma_{A}^{+}$ which only passes through vertices of 
$\Delta$.  Let $\tau_{n}$ be the above point process of asymptotic 
approach to $\Sigma_{\Delta}$, redefined as the sum of Dirac masses at 
the times $k\geq 1$, for which the first $n$-symbols of $\S^k x$ belong to 
$\Delta$ (so we have $\epsilon\sim\rho^{n}$).  Let $P(\phi)$ denote 
the pressure of $\phi$, and $P_{\pdelta}$ the pressure of the 
restriction of $\phi$ to the subsystem $\Sigma_{\Delta}$ (hence 
necessarily $P_{*}=P_{\pdelta}-P(\phi)<0$).

We prove in this paper that if $\Sigma_{\Delta}$ is an irreducible and aperiodic
subshift of finite type in its alphabet $\Delta$, then $\tau_{n}$ when 
rescaled by $e^{-nP_{*}}$ converges in law when $n\to\infty$.  The 
limit law is a marked Poisson point process of constant parameter 
measure $\lambda\pi$, where the parameters are given by
\[
	\lambda \;=\; (1-e^{P_{*}})\, \int h_{\pdelta}\, d\mu 
	\qquad\mbox{and}\qquad \pi_{j} \;=\; 
	(1-e^{P_{*}})\,e^{(j-1)P_{*}} \:,
\]
for $j\geq{1}$.  Here $h_{\pdelta}$ is the density of the 
Pianigiani-Yorke measure associated to the triple 
$(\Sigma_{A}^{+},\Sigma_{\Delta},\phi)$, see Sections~\ref{sec:sft} 
and~\ref{sec:symbl}, and~\cite{CMS}.  This result is mentioned 
in~\cite{Coe} with a sketch of the proofs and refers to the present 
paper for the complete proofs.  The problem of studying the asymptotic 
approach of two (or more) shift orbits, mentioned at the beginning, 
can then be obtained as a consequence of the above result, see the 
application after Theorems~\ref{thm:mpsymbv} and~\ref{thm:mpsymbe} 
below.

We explain in Section~\ref{sec:aper} the main ingredients of proving 
that the limit of dynamically defined hitting time point process is a 
marked Poisson point process by describing two hypotheses 
\mbox{(H.1-2)} which guarantee convergence in law.  We prove in later 
sections that these hypotheses are satisfied in our setting.  

\section{Laplace Transforms Technique}
\label{sec:aper}
We describe the details of the use of Laplace Transforms to study 
convergence in law of dynamically defined hitting times point 
processes.  We give a general exposition trying to extract the 
essential elements of the technique which we believe could have 
independent interest.  This follows closely the computations done 
in~\cite{CC} and we repeat them in full here for the sake of clarity 
and completeness and in order to make the present paper 
self-contained.

Let $(\Omega,\mu,\S)$ be an ergodic dynamical system on a standard 
Borel probability space $(\Omega,\mu)$.  Let $\Delta_{n}$ be a 
sequence of Borel sets of $\mu$-positive measure, such that 
$\mu(\Delta_{n})\to{0}$ (i.e.~a sequence of so-called 
\emph{asymptotically rare events}).  We suppose that a sequence of 
scales $c_{n}>0$ with $c_{n}\to{0}$ is chosen, to be used for 
rescaling the random times of occurrence of $\Delta_{n}$.  Define the 
point process $\tau_n$ of \emph{hitting times} in $\Delta_n$ rescaled 
by $c_{n}^{-1}$ as the map $\tau_n\colon 
{\Omega}\to{\M}_\sigma[0,\infty)$ given by
\begin{equation}
	\tau_n({\omega})\;=\;\sum_{k>0}\, \bigchi_{\Delta_{n}}({\S}^k 
	{\omega})\,\delta_{k c_{n}}\;,
	\label{eq:tau}
\end{equation}
where $\delta_{t}$ denotes Dirac measure at the point $t>0$, and 
${\M}_\sigma[0,\infty)$ denotes the Borel $\sigma$-finite measures on 
$[0,\infty)$.  The process of \emph{entrances} to $\Delta_n$ is given 
by
\[
	\tau_n^e({\omega})\;=\; \sum_{k>0}\, \bigchi_{\Delta_{n}}({\S}^k 
	{\omega})\, \bigchi_{\Delta_{n}^{c}}({\S}^{k-1} {\omega}) 
	\,\delta_{k c_n}\;,
\]
where $\Delta_{n}^{c}$ denotes the complement of $\Delta_n$.  Let 
$g\colon [0,\infty)\to\C$ be a continuous function with compact 
support.  Integrating $g$ by the point process $\tau_n$ gives rise to 
a random variable
\begin{equation}
	X_n(g)({\omega})\;=\;\sum_{k>0} \bigchi_{\Delta_{n}}({\S}^k 
	{\omega})\,g(k\, c_n)\;,
	\label{eq:xgdef}
\end{equation}
which is defined on the probability space $(\Omega,\mu)$.  
From~\cite{Nev} it is known that convergence in law of $X_n(g)$, for 
every $g$, is equivalent to convergence in law of the point 
process~$\tau_n$.  Throughout this paper we denote the expectation 
with respect to $\mu$ by
\[
	\Exp( X_n(g) )\;=\;\int X_n(g)({\omega})\,d\mu({\omega})\;.
\]

In this Section our aim is to show that under suitable hypothesis on 
$\Delta_{n}$ and $c_{n}$, the sequence of point processes $\tau_{n}$ 
converges in law to a \emph{marked Poisson point process} of constant 
parameter measure $\lambda\pi$.  Recall that such a process is defined 
by a map $\tau\colon(\Omega,\probability)\to{\M}_\sigma[0,\infty)$ 
given by
\[
	\tau(\omega)\;=\;\sum_{k>0}\, L_{k}(\omega)\, 
	\delta_{X_{k}(\omega)}\;,
\]
where for all $k,k'$, $D_{k}=X_{k}-X_{k-1}$ and $L_{k'}$ are 
independent random variables, $D_{k}$ are exponentially distributed 
with parameter $\lambda>0$ (here $X_{0}\equiv 0$) and $L_{k'}$ is a 
positive integer valued random variable with distribution 
$\pi=\{\pi_{j}\}_{j>0}$, i.e.~$\probability(L_{k'}=j)=\pi_{j}$.  (In 
the special case of $\pi_{1}=1$, $\tau$ is the \emph{Poisson point 
process} of parameter $\lambda>0$.)  Integrating $g$ with respect to 
$\tau$ we obtain a random variable $X(g)$ whose Laplace transform is 
given by
\begin{equation}
	\Exp( e^{z X(g)}) \;=\; \exp \Big\{\lambda 
	\sum_{j=1}^{\infty} \pi_j \int_{0}^{\infty} 
	(e^{z\,j\, g(y)}-1) \: 
	dy \Big\} \;.
	\label{eq:lapo}
\end{equation}

Regarding the random variable $X_{n}(g)$ of (\ref{eq:xgdef}), its 
Laplace transform admits the factorial moments expansion
\[
	\psi_{n}(z) \;=\; \Exp(e^{z X_{n}(g)}) \;=\; \sum_{k\geq 0}\, 
	\dfrac{\Exp(X_{n}(g)^{k})}{k!}\,z^{k}\:,
\]
for all $z\in\C$.  By Proposition 8.49 of Breiman~\cite{Bre} we know 
that if 
\[
\nu_{k}\;=\;\lim_{n\to\infty}\,\Exp(X_{n}(g)^{k})
\] 
exists and satisfies
\[
	\limsup_{k\to\infty}\, \dfrac{|\nu_{k}|^{1/k}}{k} \:<\: \infty\:,
\]
then there exists a random variable $X(g)$ with Laplace transform 
given by $\psi(z)=\sum_{k\geq 0}\,\nu_{k}\, z^{k}/k!$ such that 
$X_{n}(g)$ converges in law to~$X(g)$.  Hence our strategy is 
to show that the Laplace transform of~$X(g)$ coincides 
with~(\ref{eq:lapo}) in some disc around the origin of~$\C$ and 
identify the constants $\lambda$ and $\pi_{j}$.

Let $\ell(n)$ be an arbitrary sequence of positive real numbers such 
that $\ell(n)\to\infty$ as $n$ diverges.  Consider the next conditions 
on $\Delta_{n}$ and $c_{n}$.  %
\noindent {\bf H.1} (Mean Intermediate Intersections Property)
The following limit exists
\[
	C_m\;=\;\lim_{n\to\infty}\; c_{n}^{-1} 
	\sum_{\stackrel{\scriptstyle 
	{0=q_0<q_1<\cdots<q_{m-1}}}{q_s-q_{s-1}\leq \ell(n)/m}} 
	\Exp\Big(\prod_{s=0}^{m-1}\bigchi_{\Delta_n}\scirc {\S}^{q_s} 
	\Big)\;,
\]
for every fixed $m>0$, and $C_{1}>0$.  Moreover, there exist 
$\widetilde{K},\theta>0$ such that \mbox{$C_{m}\leq 
\widetilde{K}\,\theta^{m}$}.
\bigskip

\noindent {\bf H.2} (Relativised Decay of Correlations)
There exist $K_{m}>0$ and $0<\gamma<1$ such that for every 
$0=j_{0}<j_{1}<\cdots<j_{m}$ satisfying $j_{s}-j_{s-1}\leq 
\ell(n)/m$, we have for sufficiently large $n$,
\[
	\Big| \Exp\big( \prod_{s=0}^{m} 
	\bigchi_{\Delta_n}\scirc{\S}^{j_{s}} \cdot\bigchi_{B}\scirc 
	{\S}^{r+j_{m}} \big) - \Exp\big( \prod_{s=0}^{m} 
	\bigchi_{\Delta_n}\scirc{\S}^{j_{s}} \big) \,\mu(B) \Big|\;\leq\; 
	K_{m}\,\gamma^{r+j_{m}}\,\mu(B)\;,
\]
for every $r>0$, and for every Borel set 
$B\subseteq{\Omega}$.

Note that if~(H.1) is satisfied then 
$C_{1}=\lim_{n\to\infty}\,c_{n}^{-1}\mu(\Delta_{n})$ is assumed to 
exist.  Note also that the limit in~(H.1) does not depend on the 
choice of~$\ell(\cdot)$.  Hypothesis~(H.2) is a type of $\phi$-mixing 
property, combined with the fact that $\Delta_{n}$ has small measure, 
this explains the introduction of the power $j_{m}$ on the right-hand 
side.  

These properties give the following general result.  %
\begin{theorem}
If $(\Delta_{n},c_{n})$ satisfy (H.1-2) then, for any continuous 
{non-negative} function $g$ with compact support on $[0,\infty)$, the 
limit
\[
	\nu_k\;=\; \lim_{n\to\infty}\,\Exp( X_{n}(g)^{k})
\]
exists for every $k>0$, and it is the $k$-th derivative at the origin 
of the complex function
\[
	{F}_g(z)\;=\; \exp\Big\{{\sum_{m=1}^{\infty} C_{m} 
	\int_{0}^{\infty} \big(e^{z\,g(t)}-1\big)^{m}dt}\Big\}\;,
\]
which is well-defined and analytic on a disc around the origin.
\label{prop:fact}
\end{theorem}

For the proof of the above result see Appendix~\ref{app:fact}.  In 
order to identify the limit law of $\tau_{n}$ as a marked Poisson 
point process, one needs to formally solve the equation
\begin{equation}
	{\sum_{m=1}^{\infty} C_{m} \int_{0}^{\infty} \big(e^{z\, 
	g(t)}-1\big)^{m}dt}\;=\; \lambda \sum_{j=1}^{\infty} \pi_j 
	\int_{0}^{\infty} (e^{z\,j\, g(y)}-1) \: dy \;
	\label{eq:equal}
\end{equation}
in the constants $\lambda$ and $\pi_{j}$ (where $\sum\,\pi_{j}=1$), 
and ensure that the series on the right-hand side is absolutely 
convergent for $z$ in a neighbourhood of the origin.  

In Appendix~\ref{app:ident} we do this computation explicitly for the 
special case of $C_{m}=c\,\theta^{m-1}$, for some $c>0$ (this is the 
only case needed in this paper and it is referred to in 
Section~\ref{sec:symbl}), and we obtain in this case
\begin{equation}
	\lambda \;=\; \dfrac{c}{1+\theta} \qquad\mbox{and}\qquad \pi_{j} 
	\;=\; \dfrac{{\theta}^{j-1}}{(1+\theta)^{j}} \;.
	\label{eq:parm}
\end{equation}

As already pointed out in~\cite{CC} and~\cite{Coe}, using some facts 
about convergence of point processes (cf.~\cite{DV}), 
Theorem~\ref{prop:fact} together with the existence of the 
constants $\lambda$ and $\pi_{j}$ solving~(\ref{eq:equal}) prove the 
next two results.  %
\begin{theorem}
Suppose $(\Delta_{n},c_{n})$ satisfy (H.1-2).  Then, there exists 
$\lambda>0$ and a probability measure on the positive integers $\pi$ 
such that the process of hitting times $\tau_n$ in $\Delta_n$ rescaled 
by $c_{n}^{-1}$ converges in law to a marked Poisson point process 
with constant parameter measure $\lambda\,\pi$, where $(\lambda,\pi)$ 
is the solution of~(\ref{eq:equal}).
\label{thm:mpgen}
\end{theorem}
\begin{theorem}
Under the hypotheses of Theorem~\ref{thm:mpgen}, the process of 
successive entrances $\tau_n^e$ to $\Delta_n$ rescaled by $c_{n}^{-1}$ 
converges in law to a Poisson point process with parameter~$\lambda$.
\label{thm:mpgene}
\end{theorem}

\section{Symbolic Dynamics and Pianigiani-Yorke measure}
\label{sec:sft}
Let $\V=\{1,\ldots,\ell\}$ be a finite set of symbols which we will 
refer to as the base alphabet.  Throughout $A$ denotes an irreducible and aperiodic 
$0\!-\!1$ $\ell\times\ell$ matrix which defines the allowable 
transitions in a directed graph $\mathcal{G}$ of labelled vertices 
$\V$.  Define the space of one-sided allowable paths in the graph 
$\mathcal{G}$ by
\[
	\Sigma_{A}^{+}\;=\; \big\{x=(x_{n})\in\V^{\N}\colon\; 
	A(x_{i-1},x_{i})=1\,,\;\forall\,i\geq 1\big\}\;.
\]
The space $\Sigma_{A}^{+}$ is compact and metrisable when endowed with 
the Tychonov product topology (generated by the discrete topology on 
$\V$).  The \emph{shift} $\S$ (of finite type) is the map 
$\S\colon\Sigma_{A}^{+}\to\Sigma_{A}^{+}$ defined by 
$\S(x)_{n}=x_{n+1}$ for all $n\geq 0$.  This map is easily seen to be 
continuous and surjective.  The cylinders, denoted by
\[
	C[i_{0},\ldots,i_{m}]_{k}\;=\; \big\{ x\in\Sigma_{A}^{+}\colon\; 
	x_{j+k}=i_{j}\,,\;\forall\,j=0,\ldots,m \big\}\;,
\]
form a base of open (and closed) sets in $\Sigma_{A}^{+}$.  Let 
$C(\Sigma_{A}^{+})$ denote the space of complex valued continuous 
functions on $\Sigma_{A}^{+}$.  For $\psi\in C(\Sigma_{A}^{+})$, 
consider $\var_{n}(\psi) = \sup\{|\psi(x) - \psi(y)|\colon\: x_{i} = 
y_{i},\: i\leq n\}$.  Given $0 < \theta < 1$, define $|\psi|_{\theta} 
= \sup\big\{\var_{n}(\psi)/\theta^{n}\}$.  The space $\F_{\theta}^{+}= 
\{\psi\in C(\Sigma_{A}^{+})\colon\: |\psi|_{\theta} < \infty\}$ is a 
Banach space when endowed with the norm $\|\psi\|_{\theta} = 
\|\psi\|_{\infty} + |\psi|_{\theta}$, where $\|\cdot\|_{\infty}$ 
denotes the supremum norm.  The union $\F=\cup_{\theta}\,\F_{\theta}$ 
is referred to as the space of H\"older continuous functions on 
$\Sigma_{A}^{+}$.

Given a potential $\phi\in\F_{\theta}^{+}$, let  $\L_{\phi}$ be the transfer operator on $\F_{\theta}^{+}$.
It is defined as
\[
	(\L_{\phi}\psi)(x) = \sum_{\S y=x}\, 
	e^{\phi(y)}\psi(y)\;.
\]
The operator $\L_{\phi}$ has a maximum positive eigenvalue 
$e^{P(\phi)}$, which is simple and isolated.  Moreover, the rest of 
the spectrum is contained in a disc of radius strictly less than 
$e^{P(\phi)}$ (cf.~\cite{Bowen,Rue}).  The number $P= P(\phi)$ is 
called the \emph{pressure} of $\phi$.  There is a unique 
$\S$-invariant probability measure $\mu=\mu_{\phi}$ such that
\[
	P(\phi) 
	= h(\mu) + \int \phi\,d\mu\;,
\]
where $h(\mu)$ denotes the measure-theoretic entropy of $(\S,\mu)$.  
The pressure $P(\phi)$ can also be characterised as the maximum of 
$h(m) + \int\phi\,dm$ over all $\S$-invariant probabilities~$m$.  The 
measure $\mu$ is called the {\em equilibrium state} of $\phi$.  An 
eigenfunction $w$ of $\L_{\phi}$ corresponding to $e^{P(\phi)}$ may be 
taken to be strictly positive, in fact one may take $w$ to be the 
function
\[
	w \;=\; \lim_{n\to\infty}\, e^{-nP(\phi)}\,\L_{\phi}^{n}(\1) \;,
\] 
where $\1$ denotes the constant function equal to $1$.  Replacing 
$\phi$ by $\phi' = \phi - P(\phi) + \log(w) - \log(w\scirc\S)$, we see 
that $\L_{\phi'}\1=\1$ and $P(\phi')=0$.  In this case we say that 
$\phi'$ is \emph{normalised} (cf. \cite{PP}).  It is easy to see that $\phi$ and 
$\phi'$ have the same equilibrium state $\mu$.  In what follows we 
will assume that $\phi$ is normalised.  Note that in this case the 
transfer operator $\L_{\phi}$ satisfies
\begin{align}
	\int \psi \: d\mu 
	\;&=\; \int  \L_\phi(\psi) \: d\mu\,, \nonumber\\
	\int \psi_1\cdot(\psi_2 \scirc \S) \: d\mu \;&=\; \int 
	\L_\phi(\psi_1)\cdot\psi_2  \: d\mu \,,
	\label{eq:top2}
\end{align}
for all $\psi,\psi_{1},\psi_{2}\in C(\Sigma_{A}^{+})$.  

Let $\Delta\subseteq\V$ be a sub-alphabet such that $\Delta\not=\V$.  
Consider the closed $\S$-invariant subset $\Sigma_{\Delta}\subseteq 
\Sigma_{A}^{+}$ given by
\begin{equation}
	\Sigma_{\Delta}\;=\; \big\{ x\in\Sigma_{A}^{+}\colon\; 
	x_{i}\in\Delta,\;\forall\,i\geq 0 \big\}\;.
	\label{eq:sgdef}
\end{equation}

A general \emph{subsystem of finite type} in $\Sigma_{A}^{+}$ is 
obtained by prescribing a finite number of finite length allowable 
paths (of possibly different lengths) in $\mathcal{G}$ and defining 
the subsystem by the infinite allowable paths built out of finite 
pieces which respect these choices of prescribed paths.  However, 
using a higher-block representation of $\Sigma_{A}^{+}$ and choosing a 
sub-alphabet in the higher-block alphabet, by forcing the transitions 
to respect the chosen sub-alphabet we obtain the so-called ``0-step'' 
subsystem of finite type as above.  Hence there is no loss of 
generality to consider $\Sigma_{\Delta}$ as in (\ref{eq:sgdef}).

In this paper we will consider only the case when $\Sigma_{\Delta}$ is
an irreducible and aperiodic subshift of finite type in its alphabet
$\Delta$.  This means that the restriction of the matrix $A$ to the
symbols of~$\Delta$ defines a matrix $A_{\Delta}$ which is irreducible
and aperiodic.  In particular, the restriction of the shift
transformation~$T$ to $\Sigma_{\Delta}$ is topologically mixing in the
induced topology from $\Sigma_{A}^{+}$.

Let $\phi_{\pdelta}$ denote the restriction of $\phi$ to the subsystem 
$\Sigma_{\Delta}$.  Let $P_{\pdelta}$ be the pressure of 
$\phi_{\pdelta}$ with respect to the subsystem $(\Sigma_{\Delta},\S)$.  
(Note that since $\phi$ is assumed to be normalised we have 
$P(\phi)=0$, therefore $P_{\pdelta}<0$.)  Let $\mu_{\pdelta}$ denote the 
equilibrium state of $\phi_{\pdelta}$ with respect to the subsystem 
$(\Sigma_{\Delta},{\S})$.  Let $w_{\pdelta}$ be the strictly positive 
H\"older continuous function defined on~$\Sigma_{\Delta}$ by
\begin{equation}
	w_{\pdelta} \;=\; \lim_{n\to\infty}\, 
	e^{-nP_{\pdelta}}\,\L_{\phi_{\pdelta}}^{n}(\1) \;.
	\label{eq:wdel}
\end{equation}
Now define the \emph{restricted transfer 
operator} $\ldelta{}$ acting on the space of H\"older continuous 
functions $\F_{\theta}^{+}$ by
\[
	\ldelta{} \psi\;=\;{\L}_\phi(\psi\cdot\bigchi_\Delta)\;,
\]
and consider the subset of~$\Sigma_{A}^{+}$ given by
\[
	\cZ_{\Delta} \;=\; \{x\in\Sigma_{A}^{+}\colon\: 
	\exists\,b\in\Delta\,, \;A(b,x_{0})=1 \} \;.
\]
Note that since $A$ is irreducible and aperiodic in the full 
alphabet~$\V$, $\cZ_{\Delta}$ is a non-empty finite union of cylinder 
sets of~$\Sigma_{\Delta}^{+}$.  In particular, since $\mu$ is fully 
supported on~$\Sigma_{A}^{+}$ we have $\mu(\cZ_{\Delta})>0$.

An improvement to the main result of~\cite{CMS} gives the following 
result.  (See Appendix~\ref{app:cms} for a review of the main 
differences in the proof.)  %
\begin{proposition}
There exists a \emph{unique} H\"older continuous function $h_{\pdelta}$ 
defined on the \emph{whole space} $\Sigma_{A}^{+}$ such that
\[
	\L_{\Delta}(h_{\pdelta}) \;=\;  e^{P_{\pdelta} } 
	\,h_{\pdelta} \;,
\]
and $h_{\pdelta}|_{\Sigma_{\Delta}}\equiv w_{\pdelta}$, where 
$w_{\pdelta}$ is given by~(\ref{eq:wdel}).  The function $h_{\pdelta}$ 
is strictly positive on~$\cZ_{\Delta}$ and it is zero on the 
complement~$\cZ_{\Delta}^{c}$.  Moreover,
\[
	\Big\Vert e^{-nP_{\Delta}}\L_{\Delta}^n(\psi) - 
	h_{\pdelta} \int_{\Sigma_{\Delta}} \psi \: d\mu_{\pdelta} 
	\Big\Vert_{\infty} \; \underset{n \to 
	\infty}{\longrightarrow}\; 0 \;,
\]
for all $\psi \in C({\Sigma_{A}^{+}})$.
\label{prop:cms}
\end{proposition}

The Borel measure $\mu_{\ss{PY}}$ defined by
\[
	\mu_{\ss{PY}}(B)\;=\;\int_{B} h_{\pdelta}\,d\mu,
\]
for every Borel set $B\subseteq{\Sigma_{A}^{+}}$, is called the 
\emph{Pianigiani-Yorke measure} of the subsystem 
$(\Sigma_{\Delta},{\S})$.  This measure is fully supported 
on~$\cZ_{\Delta}$.  In Appendix~\ref{app:cms} we show how the main 
result of~\cite{CMS} implies the remaining statements of this section.

The measure $\mu_{PY}$ is a quasi-stationary measure satisfying
\begin{equation}
	\mu_{\scriptscriptstyle{PY}}(B) \;=\; e^{-P_{\Delta}} 
	\,\mu_{\scriptscriptstyle{PY}}(\S^{-1}B\cap \Delta)\;,
	\label{eq:mupy}
\end{equation}  
where we have identified the set of vertices $\Delta$ with the subset 
of $\Sigma_{A}^{+}$ given by 
$\{x\in\Sigma_{A}^{+}\colon\:x_{0}\in\Delta\}$.  For each $n>0$ 
consider the set
\begin{equation}
	\Delta_{n}\;=\; \big\{ x\in \Sigma_{A}^{+}\colon\;x_i\in 
	\Delta,\,\hbox{for}\:i=0,\ldots,n-1 \big\}\;.
	\label{eq:deln}
\end{equation}
We note that $\Delta_{n}\subseteq\Delta_{n-1}$ for each $n>0$ and 
$\cap_{n>0}\,\Delta_{n}=\Sigma_{\Delta}$.  The Pianigiani-Yorke 
measure also satisfies
\begin{equation}
	\frac{\mu_{\scriptscriptstyle{PY}}(B)}
	{\mu_{\scriptscriptstyle{PY}}({\Sigma_{A}^{+}})} 
	\;=\; \lim_{k \to \infty} \, \mu\big(\S^{-k} B \,|\, 
	\Delta_{k}\big) \;,
	\label{eq:cond}
\end{equation}
for all Borel sets $B$ of ${\Sigma_{A}^{+}}$, and then it carries a 
statistical information of how subsystems of finite type are embedded 
into larger systems.

\section{Limit laws for symbolic systems}
\label{sec:symbl}
Let $\Delta_{n}$ be defined by~(\ref{eq:deln}) and consider the 
hitting time point process $\tau_{n}$ of $\Delta_{n}$ rescaled by 
$c_{n}^{-1}$ of~(\ref{eq:tau}), where $c_{n}$ is to be suitably chosen 
below.  The aim is to show that Hypotheses (H.1-2) are satisfied and 
then to conclude convergence in law of $\tau_{n}$ to a marked Poisson 
point process, identifying the parameters $\lambda$ and $\pi_{j}$.

Usually $c_{n}=\mu(\Delta_{n})$ is the natural choice to rescale 
dynamically defined hitting time processes (for instance by 
conditioning the process to starting at~$\Delta_{n}$, one studies 
asymptotic return times and this scale is very natural in view of 
Ka\v{c}'s Lemma, see more comments in~\cite{Coe}).  Hence we begin with an 
asymptotic estimate on $\mu(\Delta_n)$ in terms of the relative 
pressure~$P_{\Delta}$.  Since $P_{\Delta}$ is an intrinsic constant 
associated to the triple $(\Sigma_{A}^{+},\Sigma_{\Delta},\phi)$, we 
will then take $c_{n}=e^{nP_{\Delta}}$ as our scale choice.  %
\begin{lemma}
Let $h_\Delta$ be the function in Proposition~\ref{prop:cms}.  For 
every $s\geq 0$ we have
\[
	\lim_{n\to\infty} e^{-n P_{\pdelta}}\,\mu(\Delta_{n+s})\;=\; e^{s 
	P_{\pdelta}}\int h_{\pdelta}\,d\mu\; =\; e^{s 
	P_{\pdelta}}\mu_{\ss{PY}}({{\Sigma_{A}^{+}}})\; =\; 
	\mu_{\ss{PY}}(\Delta_s)\;,
\]
where we have defined $\Delta_0={{\Sigma_{A}^{+}}}$.
\label{lem:dels}
\end{lemma}
\begin{proof}
Note that
\[
	\bigchi_{\Delta_{n+s}}\;=\; \prod_{j=0}^{n+s-1} 
	\bigchi_\Delta\scirc {\S}^j\;,
\]
for every $n>0$.  Therefore
\[
	\mu(\Delta_{n+s})\;=\;\int \bigchi_{\Delta_{n+s}}\,d\mu\;=\; 
	\int\bigchi_\Delta\cdot\bigchi_{\Delta_{n+s-1}}\scirc 
	{\S}\,d\mu\;.
\]
Since ${\L_\phi}$ satisfies property~(\ref{eq:top2}) we obtain
\[
	\mu(\Delta_{n+s})\;=\;\int \ldelta{}\1\cdot 
	\bigchi_{\Delta_{n+s-1}}\,d\mu\;,
\]
and by induction we have
\[
	\mu(\Delta_{n+s})\;= \; \int (\ldelta{n}\1)\cdot 
	\bigchi_{\Delta_s}\,d\mu\; =\; \int \ldelta{n+s}\1\,d\mu\;.
\]
Hence Proposition~\ref{prop:cms} gives
\[
	\mu(\Delta_{n+s})\;= \; e^{n P_{\pdelta}}\int_{\Delta_s} 
	h_{\pdelta}\,d\mu + o( e^{n P_{\pdelta}})\; =\;e^{(n+s) P_{\pdelta}}\int 
	h_{\pdelta}\,d\mu + o( e^{n P_{\pdelta}})\;.
\]
\vskip-15pt
\end{proof}

The next results are used to prove that Hypotheses (H.1-2) are 
satisfied.  
\begin{lemma}
There exist $K>0$ and $0<\gamma<1$ such that
\[
	\Big| \Exp( \bigchi_{\Delta_s}\cdot\bigchi_{B}\scirc 
	{\S}^{s+r} ) - \mu(\Delta_s)\,\mu(B) \Big|\;\leq\; 
	K\,\gamma^r\,e^{s P_\Delta}\,\mu(B)\;,
\]
for every $s,r>0$ and for every Borel set $B\subseteq 
{{\Sigma_{A}^{+}}}$.
\label{lem:decay}
\end{lemma}
\begin{proof}
We note that
\[
	\Exp( \bigchi_{\Delta_s}\cdot\bigchi_{B}\scirc {\S}^{s+r} 
	) \;=\; \Exp( \bigchi_{B}\cdot{\L}^r(\ldelta{s}\1) 
	)\;.
\]
From the spectral properties of ${\L}$ we know that there exists 
$0<\gamma<1$ and $K>0$ such that for every $k>0$,
\[
	\big\Vert {\L}^k w \big\Vert_\infty\;\leq\; 
	K\,\gamma^k\,\|w\|_\theta\;,
\]
whenever $w\in{\F}_\theta^{+}$ with $\int w\,d\mu=0$ (cf.~\cite{PP}).  
Since $e^{-sP_\Delta} \ldelta{s}\1$ has uniformly bounded H\"older 
norm (see~\cite{CMS} and Appendix~\ref{app:cms}), taking $w=w_s= e^{-s 
P_\Delta}(\ldelta{s}\1-\mu(\Delta_s))$, there exists $K'>0$ 
independent of $r$ and $s$ such that
\[
	\big| \Exp( \bigchi_{B}\cdot{\L}^r w_s ) \big| 
	\;\leq\; K'\, \gamma^r\,\mu(B)\;,
\]
for every $r,s>0$.  Therefore the lemma follows.
\end{proof}
\begin{lemma}
For all integer $m\geq 1$, the limit
\[
	\widetilde{C}_m\;=\;\lim_{n\to\infty}\; \mu(\Delta_n)^{-1} 
	\sum_{\stackrel{\scriptstyle 
	{0=q_0<q_1<\cdots<q_{m-1}}}{q_s-q_{s-1}\leq n/m}} 
	\Exp\Big(\prod_{s=0}^{m-1}\bigchi_{\Delta_n}\scirc {\S}^{q_s} 
	\Big)\; .
\]
Moreover,
\[
	\widetilde{C}_{m} \;=\; (e^{-P_{\Delta}}-1)^{-(m-1)} 
	\:,
\]
for all $m\geq 1$.
\label{lem:cm}
\end{lemma}
\begin{proof}
For fixed $n>0$, since $q_s-q_{s-1}\leq n/m \leq n$ for each $s$, 
we conclude that
\[
	\prod_{s=0}^{m-1}\bigchi_{\Delta_n}\scirc {\S}^{q_s} \;=\; 
	\prod_{j=0}^{q_{m-1}+n-1} \bigchi_{\Delta}\scirc {\S}^{j} \;=\; 
	\bigchi_{\Delta_{n+q_{m-1}}} \;.
\]
Write $\beta=e^{P_{\Delta}}$.  Using Lemma~\ref{lem:dels} we see that 
${\mu(\Delta_{n+q_{m-1}})}/{\mu(\Delta_n)}$ is uniformly bounded in 
$n$ by $C\eta^{q_{m-1}}$, for some $C>0$ and $0<\eta<1$, and it 
converges to $\beta^{q_{m-1}}$ as $n\to\infty$.  Hence we obtain
\[
	\widetilde{C}_m\;=\;\lim_{n\to\infty}\; 
	\sum_{\stackrel{\scriptstyle 
	{0=q_0<q_1<\cdots<q_{m-1}}}{q_s-q_{s-1}\leq n/m}} 
	\frac{\mu(\Delta_{n+q_{m-1}})}{\mu(\Delta_n)}\; \;=\; 
	\sum_{0=q_0<q_1<\cdots<q_{m-1}} \beta^{q_{m-1}}\;,
\]
and the latter summation gives $\widetilde{C}_1=1$, and 
\[
       \widetilde{C}_m\; =\; \sum_{q=m-1}^\infty 
	\frac{(q-1)\cdots(q-m+2)}{(m-2)!}\,\beta^{q} \;=\; 
	\Big(\frac{\beta}{1-\beta} \Big)^{m-1}\;\textup{for}\, m\geq 2.
\]
\vskip-15pt
\end{proof}

Lemma~\ref{lem:cm} shows that the pair $(\Delta_{n},c_{n})$ with 
$c_{n}=e^{nP_{\Delta}}$ satisfy~(H.1) with $\ell(n)=n$. Note from 
Lemma~\ref{lem:dels} that 
$c_{n}^{-1}\,\mu(\Delta_{n})\to{c}=\int{h_{\pdelta}}d\mu$ and hence
\[
	C_{m} \;=\; c\,\widetilde{C}_{m} \;=\; c\,\theta^{m-1}\:,
\]
where $\theta=(e^{-P_{\Delta}}-1)^{-1}$.  Furthermore, 
Lemma~\ref{lem:decay} shows that~(H.2) is also satisfied.  
Consequently, we may apply Theorems~\ref{thm:mpgen} 
and~\ref{thm:mpgene} together with the parameters given 
in~(\ref{eq:parm}) to obtain 
\begin{theorem}
If $\Sigma_{\Delta}$ is an irreducible and aperiodic subshift of finite type of 
$(\Sigma_{A}^{+},\S)$ then the process of hitting times $\tau_n$ in 
$\Delta_n$ scaled by $e^{-nP_{\pdelta}}$ converges in law to a marked 
Poisson point process with constant parameter measure 
$\lambda\,\pi$, where the parameters are given by
\[
	\lambda \;=\; (1-e^{P_{\pdelta}})\, \int h_{\pdelta}\, d\mu 
	\qquad\mbox{and}\qquad \pi_{j} \;=\; 
	(1-e^{P_{\pdelta}})\,e^{(j-1)P_{\pdelta}} \:,
	\label{eq:const}
\]
for $j\geq{1}$, where $P_{\pdelta}$ is the pressure of the restriction of 
$\phi$ to the subsystem $\Sigma_{\Delta}$ and $h_{\pdelta}$ is the 
density of the Pianigiani-Yorke measure associated to the triple 
$(\Sigma_{A}^{+},\Sigma_{\Delta},\phi)$.
\label{thm:mpsymbv}
\end{theorem}
\begin{theorem}
Under the hypotheses of Theorem~\ref{thm:mpsymbv}, the process of 
successive entrances $\tau_n^e$ to $\Delta_n$ scaled by 
$e^{-nP_{\Delta}}$ converges in law to a Poisson point process with 
parameter~$\lambda$ as above.
\label{thm:mpsymbe}
\end{theorem}

\bigskip

{\bf An application}. The above results give an interesting application as follows.  Suppose 
we are given $N$ points at random $\omega^{(i)}\in\Sigma_{A}^{+}$, 
independently of one another, and $\omega^{(i)}$ distributed according 
to the equilibrium state $\mu_{i}$.  Record the times $k>0$ of $n$ 
matchings in a row of all the sequences $\omega^{(i)}$, i.e.~consider 
the times $k>0$ such that $\omega^{(i)}_{k+s}=\omega^{(i')}_{k+s}$ for 
$s=0,\ldots,n-1$, and all $i,i'\in\{1,\ldots,N\}$.  Then rescaling the 
corresponding point process by the probability of the event hitting 
$n$ matchings in a row over all the sequences gives, in the limit, a 
marked Poisson point process of constant parameter measure.  To 
identify the parameters one only needs to consider the product shift 
of finite type $\Sigma_{A}^{+}\times\cdots\times\Sigma_{A}^{+}$ and 
notice that $\mu_{1}\times\cdots\times\mu_{N}$ is the equilibrium 
state of the potential
\[
	\phi(x^{(1)},\ldots,x^{(N)}) \;=\; \phi_{1}(x^{(1)}) +\cdots + 
	\phi_{N}(x^{(N)}) \;,
\]
where $\phi_{i}$ is the potential defining $\mu_{i}$.  Consider the 
subshift $\Sigma_{\Delta}$ as the diagonal subshift obtained by 
setting $x^{(i)}=x^{(i')}$ for all $i,i'$.  Let $\Delta_{n}$ be the 
subset of $\Sigma_{A}^{+}\times\cdots\times\Sigma_{A}^{+}$ consisting 
of the points $(x^{(1)},\ldots,x^{(N)})$ such that 
$x^{(i)}_{s}=x^{(i')}_{s}$ for $s=0,\ldots,n-1$, and all 
$i,i'\in\{1,\ldots,N\}$.  We apply the above results to this situation 
and conclude that the parameters $\lambda$ and $\pi$ of the limiting 
point process is given by~(\ref{eq:const}), where $P_{\pdelta}$ is 
replaced by $P_{*}=P_{\phi_{\pdelta}}-P(\phi)$, $P(\phi)$ denotes the 
pressure of $\phi$ on $\Sigma_{A}^{+}\times\cdots\times\Sigma_{A}^{+}$ 
and $P_{\phi_{\pdelta}}$ is the pressure of the restriction 
$\phi_{\pdelta}$ of the potential $\phi$ to the subshift 
$\Sigma_{\Delta}$.  In the special case of two orbits we also have an 
analogue of our main result in~\cite{CC} as follows.  Define a 
distance on $\Sigma_{A}^{+}$ by
\[
	d(x,y) \;=\; \sum_{k\geq{0}}\,|x_{k}-y_{k}|\,\rho^{k} \:,
\]
for $\rho=e^{P_{*}}$.  We see that given $\epsilon>0$ the point 
process $\tau_{\epsilon}$ obtained by summing Dirac point masses at 
the times $k\geq1$ such that $d(T^{k}x,T^{k}y)\leq \epsilon$ converges in 
law, when rescaled by $\epsilon^{-1}\sim e^{-nP_{*}}$, to a marked 
Poisson point process of constant parameter measure $\lambda\pi$ as 
above.

\appendix
\section{Convergence of Factorial Moments}
\label{app:fact}
Here we reproduce the computations of~\cite{CC} adapted to our general 
setting. 
\begin{proof}[Proof of Theorem~\ref{prop:fact}]
We would like to show that 
$\nu_{k}=\lim_{n\to\infty}\,\Exp(X_{n}(g)^{k})$ exists and that 
$\psi(z)=\sum_{k\geq 0}\,\nu_{k}\, z^{k}/k!$ takes the form
\begin{equation}
	\psi(z)\;=\; {F}_g(z)\;=\; \exp\Big\{{\sum_{m=1}^{\infty} \, C_{m} 
	\int_{0}^{\infty} \big(e^{z\,g(t)}-1\big)^{m}dt}\Big\}\;,
\end{equation}
for $z$ in a small disc around the origin, where $C_{m}$ is defined 
in~(H.1).  Since the $k$-th derivative of ${F}_g$ at the origin is 
given by
\begin{equation}
\begin{split}
	&\sum_{p=1}^{k}\: \sum_{\stackrel{\scriptstyle 
	{0<t_{1},\ldots,0<t_{p}}} {\scriptstyle 
	t_{1}+\cdots+t_{p}=k}}{\frac{k!}{t_{1}!\cdots t_{p}!}} \: 
	\sum_{b=1}^{p} \: 
	\sum_{\stackrel{\scriptstyle{0<n_{1},\ldots,0<n_{b}}}{\scriptstyle 
	n_{1}+\cdots+n_{b}=p}} \: \prod_{i=1}^{b} \, 
	C_{n_{i}}\quad\times \\
	&\int_{0}^{\infty} dy_{b}\, g(y_{b})^{\sum_{s=0}^{n_{b}-1} 
	t_{s+n_{1}+\cdots+n_{b-1}+1}} \ldots \int_{0}^{y_{2}} dy_{1}\, 
	g(y_{1})^{\sum_{s=0}^{n_{1}-1}t_{s+1}} \;, \\
\end{split}\label{eq:lim}
\end{equation}
we want to prove that $\nu_{k}$ can be expressed by (\ref{eq:lim}) for 
every $k>0$.

For $k=1$ we have
\[
	\Exp( X_{n}(g) )\;=\; \Exp( \bigchi_{\Delta_n} ) 
	\sum_{s=1}^{\infty} g(s\,c_{n}) \;=\; c_{n}^{-1}\mu(\Delta_n) \;\; 
	c_{n} \sum_{s=1}^{\infty} g(s\,c_{n})\;.
\]
By~(H.1) we know that $c_{n}^{-1}\mu(\Delta_{n})$ converges to 
$C_{1}>0$.  Since $g$ is continuous with compact support we obtain
\[
	\lim_{n\to\infty}\:\Exp( X_{n}(g) ) \;=\; C_{1} 
	\int_{0}^{\infty}{g(y)}\,dy\;.
\]
Now suppose $k>1$.  We know that
\[
	\Exp( X_{n}(g)^{k} )\;= \;\sum_{0\le j_{1},\ldots,0\le j_{k}} 
	\Exp\Big( \prod_{s=1}^{k}g(j_{s}\,c_n) \, 
	\bigchi_{\Delta_n}\!{\scirc} {\S}^{j_{s}} \Big)\;,
\]
and rearranging this sum we obtain
\[
	\Exp( X_{n}(g)^{k} )\;=
\]
\[	
	\sum_{p=1}^{k}\;\; 
	\sum_{\stackrel{\scriptstyle {0<t_{1},\ldots,0<t_{p}}} 
	{\scriptstyle t_{1}+\cdots+t_{p}=k}}{\frac{k!}{t_{1}!\cdots 
	t_{p}!}}\;\; \sum_{0\le j_{1}<j_{2}<\cdots<j_{p}}\! 
	\Exp\Big(\prod_{s=1}^{p}g(j_{s}\,c_n)^{t_{s}}\, 
	\bigchi_{\Delta_n}\!{\scirc} {\S}^{j_{s}} 
	\Big)\;.
\]
For a fixed set of positive integers $t_1,\ldots,t_p$ define the 
summation
\[
	S(\{t_i\},n)\;=\;\sum_{0\le j_{1}<j_{2}<\cdots<j_{p}}\, 
	\Exp\Big(\prod_{s=1}^{p}g(j_{s}\,c_n)^{t_{s}}\, 
	\bigchi_{\Delta_n}\!{\scirc} {\S}^{j_{s}} 
	\Big)\;.
\]
Decomposing this sum into clusters of consecutive indices differing by 
at most $\ell(n)/k$ we obtain
\[
	S(\{t_i\},n)\;=\; \sum_{b=1}^{p}\;\; \sum_{\stackrel{\scriptstyle 
	{0<n_{1},\ldots,0<n_{b}}}{n_{1}+\cdots+n_{b}=p}}\;\; 
	\sum_{(j_{1},\ldots,j_{p})\in {\Q}(n_{1},\ldots,n_{b})} 
	\Exp\Big(\prod_{s=1}^{p}g(j_{s}\,c_n)^{t_{s}} 
	\,\bigchi_{\Delta_n}\!{\scirc} 
	{\S}^{j_{s}}\Big)\;,
\]
where we have defined
\[
{\Q}(n_{1},\ldots,n_{b})=
\]
\[
\Bigl\{(j_{1},\ldots,j_{n_{1}+ \cdots+n_{b}})\,|\;\,j_{1}<j_{2}<\cdots<j_{n_{1}+\cdots+n_{b}},
	 j_{q+1}-j_{q}\le \ell(n)/k
\]
\[	 
\textup{if}\;q\notin \{n_{1},n_{1}+n_{2},\ldots,n_{1}+\cdots+n_{b-1}\}
\;	\textup{and else}\; j_{q+1}-j_{q}> \ell(n)/k\Bigr\}\;. 
\]
Now we use the `relativised' decay of correlations (condition H.2) between the 
different clusters.  Fixing the positive integers $n_1,\ldots,n_b$ and 
fixing $(j_1,\ldots,j_{n_1+\cdots+n_b})\in {\Q}(n_1,\ldots,n_b)$ we 
have
\[
	\Exp\Big(\prod_{s=1}^{n_{1}+\cdots+n_{b}}\, 
	\bigchi_{\Delta_n}\!{\scirc} {\S}^{j_{s}}\Big)\,=\, 
\]
\[	
	\Exp\Big(\prod_{s=1}^{n_{1}}\,\bigchi_{\Delta_n}\!{\scirc} 
	{\S}^{j_{s}} 
	\prod_{s=n_{1}+1}^{n_{1}+n_{2}}\,\bigchi_{\Delta_n}\!{\scirc} 
	{\S}^{j_{s}} \;\;\cdots \!\!\!\!\!\!\!{
	\prod_{s=n_{1}+\cdots+n_{b-1}+1}^{n_{1}+\cdots+n_{b}}\, 
	\bigchi_{\Delta_n}\!{\scirc} {\S}^{j_{s}} \Big)\;.}
\]
Therefore we can write
\[
	\Exp\Big( \prod_{s=1}^{n_{1}+\cdots+n_{b}}\, 
	\bigchi_{\Delta_n}\!{\scirc} {\S}^{j_{s}} \Big)\;=\; 
	\Exp\Big( 
	\prod_{s=1}^{n_{1}}\, 
	\bigchi_{\Delta_n}\!{\scirc} {\S}^{j_{s}-j_{1}}
	\cdot 
	\bigchi_{B}{\scirc} {\S}^{j_{n_1+1}-j_1} \Big)\;,
\]
where $B$ is a finite union of pre-images under $T$ of $\Delta_{n}$.  
Note that $j_{n_1+1}-j_{1}>j_{n_{1}}$, therefore we can apply~(H.2) to 
get
\[
	\Exp\Big( 
	\prod_{s=1}^{n_{1}+\cdots+n_{b}}\,\bigchi_{\Delta_n}\!{\scirc} 
	{\S}^{j_{s}} \Big)\;=\; \Big( 
	\Exp\Big(\prod_{s=1}^{n_{1}}\,\bigchi_{\Delta_n}\!{\scirc} 
	{\S}^{j_{s}-j_1}\Big) + R(j_1,\ldots,j_{n_1+1}) 
	\Big)\,\Exp(\bigchi_{B})\;,
\]
where the remainder $R$ satisfies
\begin{equation}
	\big| R(j_1,\ldots,j_{n_1+1}) \big|\;\leq\; 
	K_{n_{1}}\,\gamma^{j_{n_1+1}-j_{1}}\;.
	\label{eq:rem}
\end{equation}
Now using induction on the remaining clusters we obtain for $b>1$,
\begin{equation}
\Exp\Big(\prod_{s=1}^{n_{1}+\cdots+n_{b}}\, 
\bigchi_{\Delta_n}\!{\scirc} {\S}^{j_{s}}\Big)\;= 
\label{eq:clus}
\end{equation}
\[
	\Big(\Exp\Big({\prod_{s=1}^{n_{1}}\,\bigchi_{\Delta_n}\!{\scirc} 
	{\S}^{j_{s}-j_{1}}\Big)}+R({j_{1},\ldots,j_{n_{1}+1}})\Big) 
	\;\times \cdots 
\]	
\[
	\cdots\times\;\Big(\Exp\Big({ 
	\prod_{s=n_{1}+\cdots+n_{b-2}+1}^{n_{1}+\cdots+n_{b-1}}\, 
	\!\!\!\!\!\!\!\bigchi_{\Delta_n}\!{{\scirc}} {\S}^{j_{s}} 
	\Big)}+R({j_{n_{1}+ 
	\cdots+n_{b-2}+1},\ldots,j_{n_{1}+\cdots+n_{b-1}+1}})\Big) 
\]
\[
	\times\;\Exp\Big({\prod_{s=n_{1}+\cdots+n_{b-1}+1}^{n_{1}+ 
	\cdots+n_{b}}\!\!\!\!\!\!\!\!  \bigchi_{\Delta_n}\!{\scirc} 
	{\S}^{j_{s}}\Big)}\;.
\]
The above expression implies that
\[
\begin{split}
	\Exp(X_{n}(g)^{k}) \;=\; \sum_{p=1}^{k}\;\; 
	\sum_{\stackrel{\scriptstyle {0<t_{1},\ldots,0<t_{p}}} 
	{\scriptstyle t_{1}+\cdots+t_{p}=k}}{\frac{k!}{t_{1}!\cdots 
	t_{p}!}} \sum_{b=1}^{p}\;\; \sum_{\stackrel{\scriptstyle 
	{0<n_{1},\ldots,0<n_{b}}}{n_{1}+\cdots+n_{b}=p}}\;\; 
	\sum_{(j_{1},\ldots,j_{p})\in {\Q}(n_{1},\ldots,n_{b})} \\
	\prod_{m=1}^{b}\left( 
	\Exp\Big(\prod_{s=n_{1}+\cdots+n_{m-1}+1}^{n_{1}+\cdots+n_{m}} 
	\!\!\!\!\!\!\!\bigchi_{\Delta_{n}}\!{\scirc} {\S}^{j_{s}}\Big) 
	\prod_{s=n_{1}+\cdots+n_{m-1}+1}^{n_{1}+\cdots+n_{m}} 
	g^{t_{s}}(j_{s}\,c_n) \right) \;+\; \mathcal{R}(n,k), \\
\end{split}
\]
where we have set $n_{0}=0$, and $\mathcal{R}(n,k)$ is a remainder 
term.  For fixed indices $p$, $t_{1},\ldots,t_{p}$, $b$, 
$n_{1},\ldots,n_{b}$ and for 
$(j_{1},\ldots,j_{p})\in\mathcal{Q}(n_{1},\ldots,n_{b})$, define a 
double sequence of integers $(q_{m,s})$ with $1\le m\le b$ and $0\le 
s\le n_{m}-1$ by
\[
	q_{m,s}=j_{s+n_{1}+\cdots+n_{m-1}+1} - 
	j_{n_{1}+\cdots+n_{m-1}+1}\;.
\]
We then obtain

\begin{equation}
\label{eq:big}
	\sum_{(j_{1},\ldots,j_{p})\in \mathcal{Q}(n_{1},\ldots,n_{b})}\;\; 
	\prod_{m=1}^{b}\left( 
	\Exp\Big(\prod_{s=n_{1}+\cdots+n_{m-1}+1}^{n_{1}+\cdots+n_{m}} 
	\!\!\!\!\!\!\!\bigchi_{\Delta_{n}}\!{\scirc} {\S}^{j_{s}}\Big) 
	\prod_{s=n_{1}+\cdots+n_{m-1}+1}^{n_{1}+\cdots+n_{m}} 
	g^{t_{s}}(j_{s}\,c_n)\right) \;=
\end{equation}	
\[	
	\sum_{\stackrel{\scriptstyle 
	0<q_{i,1}<\cdots<q_{i,n_{i}-1}} {
	\stackrel{\scriptstyle i=1,\ldots,b}
	{\scriptstyle q_{i,s+1}-q_{i,s}\le \ell(n)/k}}} 
	c_{n}^{-b} \prod_{i=1}^{b} \: 
	\Exp\Big(\prod_{s=0}^{n_{i}-1}\, \bigchi_{\Delta_{n}}\!{\scirc} 
	{\S}^{q_{i,s}} \Big) \quad\times
\]
\[	
	\sum_{\stackrel{\scriptstyle j_{1}<j_{n_{1}+1}<\cdots< 
	j_{n_{1}+\cdots+n_{b-1}+1}} {\stackrel{\scriptstyle 
	j_{n_{1}+\cdots+n_{r+1}+1}-j_{n_{1}+\cdots+n_{r}+1}>}{\scriptstyle q_{r,n_{r}-1}+ 
	\ell(n)/k}}}	
	c_{n}^{b} \prod_{i=1}^{b} \prod_{s=0}^{n_{i}-1} 
	g^{t_{s+n_{1}+\cdots+n_{i-1}+1}} 
	((j_{n_{1}+\cdots+n_{i-1}+1}+q_{i,s}) \,c_n) \;.
\] 
From~(H.1) and the elementary properties of the Riemann integral we see 
that the expression~(\ref{eq:big}) converges to
\[
	\prod_{i=1}^{b} \, C_{n_{i}}\int_{0}^{\infty} dy_{b}\, 
	g(y_{b})^{\sum_{s=0}^{n_{b}-1} t_{s+n_{1}+\cdots+n_{b-1}+1}} 
	\;\cdots\; \int_{0}^{y_{2}} dy_{1}\, 
	g(y_{1})^{\sum_{s=0}^{n_{1}-1}t_{s+1}} \;.
\]
Hence the proof of Theorem~\ref{prop:fact} is finished, provided 
we show that $\mathcal{R}(n,k)=o(n)$, for each fixed $k$.  

We illustrate the estimation of the remainder in the case $b=2$ for 
fixed $n_{1}$ and $n_{2}$.  (The general case can be obtained in a 
similar manner as indicated at the end of this proof.)  In this case, 
the remainder is composed by a finite sum of expressions of the form
\[
	\sum_{(j_{1},\ldots,j_{p})\in {\Q}(n_{1},n_{2})} \!\!\!
	R(j_{1},\ldots,j_{n_{1}}) \;
	\Exp\Big(\prod_{s=n_{1}+1}^{n_{1}+n_{2}} 
	\bigchi_{\Delta_{n}}\!{\scirc} {\S}^{j_{s}}\Big)
	\prod_{s=1}^{n_{1}+n_{2}} 
	g^{t_{s}}(j_{s}\,c_n) \;.
\]
Using~(\ref{eq:rem}) this term is bounded by
\begin{equation}
\label{eq:bd1}
	(\|g\|_{\infty})^{n_{1}}\, K_{n_{1}}\, \times
\end{equation}
\[	
	\sum_{(j_{1},\ldots,j_{p})\in {\Q}(n_{1},n_{2})} \!\!\!
	\gamma^{j_{n_{1}+1}-j_{n_{1}}} \;
	\Exp\Big(\prod_{s=n_{1}+1}^{n_{1}+n_{2}} 
	\bigchi_{\Delta_{n}}\!{\scirc} {\S}^{j_{s}}\Big) \left|
	\prod_{s=n_{1}+1}^{n_{1}+n_{2}} 
	g^{t_{s}}(j_{s}\,c_n) \right| \;.
\]	
Since $j_{q+1}-j_{q}\leq \ell(n)/k$ for all $q<n_{1}$, when we perform 
the sum over the indices $j_{1},\ldots,j_{n_{1}}$ we obtain a factor of
$(n_{1} \ell(n)/k)^{n_{1}}$. Using the 
fact that $j_{n_{1}+1}-j_{n_{1}}> \ell(n)/k$ and introducing the 
variables $q_{m,s}$ with $m=2$ as before, we see that~(\ref{eq:bd1}) 
is bounded by
\begin{equation}
\begin{split}
\label{eq:bd2}
	(\|g\|_{\infty})^{k}\, K_{n_{1}}\; \times
         \left( \dfrac{n_{1} \ell(n)}{k} 
	\right)^{n_{1}} \,
	\sum_{s>\ell(n)/k}\, \gamma^{s} \;\;\times \hspace{1in}\mbox{} \\
	\sum_{\stackrel{\scriptstyle 
	0<q_{2,1}<\cdots<q_{2,n_{2}-1}} {
	{\scriptstyle q_{2,s+1}-q_{2,s}\le \ell(n)/k}}} \!\!\!
	\Exp\Big(\prod_{s=0}^{n_{2}-1} 
	\bigchi_{\Delta_{n}}\!{\scirc} {\S}^{q_{2,s}}\Big) \,
	\sum_{j_{n_{1}+1}}\,
	\left|
	\prod_{s=0}^{n_{2}-1} 
	g^{t_{s+n_{1}+1}}((j_{n_{1}+1}+q_{2,s})\,c_n) \right| .
\end{split}
\end{equation}
When $n$ diverges, the second part of~(\ref{eq:bd2}) is bounded by an 
integral (multiplying and dividing by $c_{n}$), whereas the first part 
of~(\ref{eq:bd2}) is of the order
\[
 	\left( \dfrac{n_{1} \ell(n)}{k} 
	\right)^{n_{1}} \, \gamma^{\ell(n)/k} \,,
\]
which clearly tends to zero as $n$ diverges because $n_{1}$ is bounded. 
 
Now, for the general estimate of the remainder, we note that 
equation~(\ref{eq:clus}) shows that $\mathcal{R}(n,k)$ is a sum of 
products of $b$ terms of the form $\Exp\big(\prod 
\bigchi_{\Delta_{n}}\!{\scirc} {\S}^{\ \cdot}\big)$ or $R(\dots)$ and 
there is at least one of the latter type.  Introducing the indices 
$m$, $s$ and $q_{m,s}$ as before, the summation over the indices can 
be performed similarly.  Hence one obtains an expression very similar 
to equation~(\ref{eq:big}) except that one multiplies and divides by a 
power of $c_n$ which equals the number of factors of the form 
$\Exp\big(\prod \bigchi_{\Delta_{n}}\!{\scirc} {\S}^{\ \cdot}\big)$.  
Using the analogous estimates as in~(\ref{eq:rem}) for the terms 
$R({\dots})$ one readily sees that the remainder is~$o(n)$.
\end{proof}

\section{Identification of the parameters}
\label{app:ident}
Here we assume $C_{m}=c\,\theta^{m-1}$ for some $c,\theta>0$, and we 
find an explicit solution (in the constants $\lambda$ and $\pi_{j}$, 
where $\sum\,\pi_{j}=1$) of
\begin{equation}
	{\sum_{m=1}^{\infty} \, C_{m} \int\big(e^{z\, 
	g(t)}-1\big)^{m}dt}\;=\; \lambda \sum_{j=1}^{\infty} \pi_j 
	\int_{0}^{\infty} (e^{z\,j\, g(y)}-1) \: dy \;,
	\label{eq:eq2}
\end{equation}
for $z$ in some disc around the origin in $\C$.  Consider the analytic 
function $\Phi(u) = \sum_{m=1}^\infty \, C_m u^m$ defined in a 
neighbourhood of the origin.  Setting $u=(1+u)-1$ and using analytic 
continuation we obtain
\begin{equation}
\begin{split}
	\Phi(u) \;&=\; \dfrac{c}{\theta}
	\,\sum_{m=1}^{\infty} \, (\theta\,u)^{m} \;=\; 
	\dfrac{c\,u}{1-\theta\,u} 
	\;=\; \dfrac{c\,u}{(1+\theta) - \theta\,(1+u)} \\
	&=\; c\, (1-\theta')\, \big[(1+u) - 1\big] \, 
	\sum_{j=0}^{\infty} \, \big( \theta'\, 
	(1+u)\big)^{j} \:,\\
\end{split} \label{eq:psc}
\end{equation}
where we have introduced $\theta'=\theta/(1+\theta)$.
Since finding a formal 
solution for $\lambda$ and $\pi_{j}$ such that
\[
	\Phi(u)\;=\; \sum_{m=1}^\infty \, C_m u^m \;=\; 
	\lambda\Big(\sum_{j=1}^\infty \pi_j(1+u)^j -1 \Big)
\]
gives a formal solution of (\ref{eq:eq2}), we use~(\ref{eq:psc}) to 
compare the coefficients of $(1+u)^{j}$ to conclude that
\[
	\lambda \;=\; c\,(1-\theta') \;=\; \dfrac{c}{1+\theta} 
	\qquad\mbox{and}\qquad \pi_{j} \;=\; (1-\theta')\,(\theta')^{j-1} 
	\;=\; \dfrac{{\theta}^{j-1}}{(1+\theta)^{j}} \;.
\]
Hence the right-hand side of~(\ref{eq:eq2}) is an analytic function in 
a neighbourhood of the origin.

\section{Eigenfunctions for the restricted transfer operator}
\label{app:cms}
This is a review of the paper~\cite{CMS} with comments on some 
improvements of their main result, which are used in the present 
paper.  

The main difference between our setting and the one used in~\cite{CMS} 
is the fact that in~\cite{CMS} there is a fixed initial finite 
alphabet $S$ (our set of vertices $\V$) and the whole space is a 
subshift of finite type $X_{L'}$ of $X=S^{\N}$ defined by an 
irreducible and aperiodic transition matrix $L'$ in the alphabet $S$ 
(our subshift $\Sigma_{A}^{+}$).  Then the authors consider a 
subsystem of $X_{L'}$ given by a transition matrix $L$ in the alphabet 
$S$, where $L$ imposes more restrictions than $L'$ (i.e.~if 
$L=[\ell_{ij}]$ and $L'=[\ell_{ij}']$, then $\ell_{ij}'=0$ implies 
$\ell_{ij}=0$).  The important thing is that $L$ is assumed to be 
irreducible and aperiodic in the full alphabet $S$, therefore the 
allowable paths of the corresponding subshifts $X_{L}$ and $X_{L'}$ go 
through all the symbols of the initial alphabet $S$.  In our setting, 
we choose a strictly smaller alphabet $\Delta\subsetneq\V$ and 
consider the allowable paths of $\Sigma_{A}^{+}$ which go through 
vertices of $\Delta$ defining then a subshift $\Sigma_{\Delta}$ with 
alphabet $\Delta$.  If we assume now that $\Sigma_{\Delta}$ is 
irreducible and aperiodic in its alphabet $\Delta$, then there may not 
exist a strictly positive eigenfunction of the restricted transfer 
operator $\L_{\Delta}$ associated to the eigenvalue $e^{P_{\Delta}}$, 
in contrast with~\cite{CMS}, where the restricted transfer operator 
$\underline{\L}$ of $X_{L}$ is shown to have a strictly positive 
eigenfunction associated to the corresponding eigenvalue $\alpha_{L}$.  
The following provides an example.  Take the set of vertices 
$\V=\{1,2,3,4\}$ and the matrix $A$ given by
\[
	A \;=\; \left(\begin{array}{cccc}
	1 & 1 & 0 & 0 \\
	1 & 1 & 1 & 0 \\
	0 & 1 & 1 & 1 \\
	0 & 0 & 1 & 1
	\end{array}
	\right)\;.
\]
Let $\Delta=\{1,2\}$ and then $\Sigma_{\Delta}$ is the full two shift 
on the symbols $\{1,2\}$.  Any function $\psi$ defined on 
$\Sigma_{A}^{+}$ satisfies $\L_{\Delta}(\psi)(x)=0$ whenever 
$x_{0}=4$, therefore $\L_{\Delta}$ does not have a strictly positive 
eigenfunction.
In fact, in general if 
$\cZ_{\Delta}$ is the subset of $\Sigma_{A}^{+}$ given by
\begin{equation}
	\cZ_{\Delta} \;=\; \{x\in\Sigma_{A}^{+}\colon\: 
	\exists\,b\in\Delta\,, \; A(b,x_{0})=1 \} \;,
	\label{eq:cz}
\end{equation}
then for any function $\psi$, $\L_{\Delta}(\psi)(x)=0$ whenever 
$x\not\in\cZ_{\Delta}$.  Assuming $\Sigma_{\Delta}$ is irreducible and 
aperiodic in its alphabet $\Delta$, the next comments show that 
$\L_{\Delta}$ has an eigenfunction associated to $e^{P_{\Delta}}$, 
which is strictly positive on $\cZ_{\Delta}$ and it is zero on the 
complement~$\cZ_{\Delta}^{c}$.

Since we would not want to rewrite the paper~\cite{CMS}, we will only 
mention the main differences.  For 
$x\in\Sigma_{A}^{+}\setminus\Sigma_{\Delta}$, let $N(x)=\inf\{n\geq 
0\colon\: x_{n}\not\in\Delta\}$.  Fix some point 
$z\in\Sigma_{\Delta}$.  Using the fact that $A$ is irreducible and 
aperiodic there exists $q>0$ such that $A^{q}$ is a strictly positive 
matrix.  This means that for any symbol $s\in\V$ there exists an 
allowable path of length $q$ in the graph of $\Sigma_{A}^{+}$ which 
starts at $s$ and ends at $z_{0}$.  Let 
$s\to\psi_{1}(s)\to\cdots\to\psi_{q-1}(s)\to z_{0}$ be such a path, 
where $\psi_{i}(s)\in\V$, for $i=1,\ldots,q-1$.  Define 
$\pi\colon\Sigma_{A}^{+}\to\Sigma_{\Delta}$ by $\pi(x)=x$ if 
$x\in\Sigma_{\Delta}$, and for 
$x\in\Sigma_{A}^{+}\setminus\Sigma_{\Delta}$ define
\[
	\pi(x) \;=\; (x_{0},\ldots, x_{N(x)}, \psi_{1}(x_{N(x)}), \ldots, 
	\psi_{q-1}(x_{N(x)}), z_{0}, z_{1}, \ldots) \;.
\]

Let $C_{p}^{+}(\Sigma_{A}^{+})$ be the 
set of strictly positive $p$-cylindrical functions (i.e.~a function 
depending only on the first $p$ coordinates of the point).  Let 
$0<\theta<1$ be the H\"older exponent of the potential $\phi$.  Let 
$\cZ_{\Delta}$ be defined as in~(\ref{eq:cz}).  %
\begin{alemma}
There exists $c>0$ such that for any $p>0$, for any $k\geq p$, and for 
any $f\in C_{p}^{+}(\Sigma_{A}^{+})$, we have 
\[
	e^{-c\,\theta^{N(x)}} \;\leq\; 
	\dfrac{\L_{\Delta}f(x)}{\L_{\Delta}f(\pi(x))} \;\leq\; 
	e^{c\,\theta^{N(x)}} \;,
\]
for all $x\in\cZ_{\Delta}$; and $\L_{\Delta}f(x)=0$ if 
$x\not\in\cZ_{\Delta}$.
\end{alemma}
The proof of the above Lemma is exactly the same as the proof of 
Lemma~1 of~\cite{CMS}.  For the next result, we note that if $f\in 
C_{p}^{+}(\Sigma_{\Delta})$ then $f$ can be extended in a 
natural way to a function defined on 
$\Delta_{p}=\{x\in\Sigma_{A}^{+}\colon\: 
x_{i}\in\Delta\,,\;i=0,\ldots,p-1 \}$.  %
\begin{alemma}
There exists $0<r<1$ and $c(f)>0$ such that for any $n>2p$, and for 
any $f\in C_{p}^{+}(\Sigma_{\Delta})$, we have
\[
	e^{-c(f)\,r^{n}} \;\leq\; e^{-P_{\Delta}}\,
	\dfrac{\L_{\Delta}^{n+1}f(x)}{\L_{\Delta}^{n}f(x)} \;\leq\; 
	e^{c(f)\,r^{n}} \;,
\]
for all $x\in\cZ_{\Delta}\subseteq\Sigma_{A}^{+}$.
\end{alemma}
Again the proof of the above Lemma is exactly the same as the proof of 
Lemma~2 of~\cite{CMS}.  Let $C({\cZ_{\Delta}})$ denote the 
set of continuous functions defined on $\cZ_{\Delta}$.  %
\begin{alemma}
For any $f\in\cup_{p\geq 1}\,C_{p}^{+}(\Sigma_{A}^{+})$, we 
have
\begin{itemize}
\item[(i)] $\{e^{-nP_{\Delta}}\,\L_{\Delta}^{n}f\}_{n\geq 0}$ is a 
Cauchy sequence in $C(\Sigma_{A}^{+})\:$;
\vspace{5pt}
\item[(ii)] $h_{\Delta} = \d\lim_{n\to\infty}\, 
\dfrac{e^{-nP_{\Delta}}\,\L_{\Delta}^{n}f}{\int f\,d\mu_{\Delta}}$ 
does not depend on the function $f\in\cup_{p\geq 
1}\,C_{p}^{+}(\Sigma_{A}^{+})$ and it satisfies 
\[
	\L_{\Delta}(h_{\Delta}) \;=\; e^{P_{\Delta}}\, h_{\Delta} \;.
\]

\end{itemize}
\label{lem:cms3}
\end{alemma}
The above Lemma is the same as Lemma~3 of~\cite{CMS}.  The proof of 
Lemma~3 of~\cite{CMS} implies that 
$\{e^{-nP_{\Delta}}\,\L_{\Delta}^{n}f\}_{n\geq 0}$ is a Cauchy 
sequence in $C({\cZ_{\Delta}})$.  Since on the complement 
$\cZ_{\Delta}^{c}$ the sequence is identically zero, we conclude that 
(i) holds.  The proof of~(ii) is the same as the proof of Lemma~3~(ii) 
in~\cite{CMS}.  We note that this proof implies that $h_{\Delta}$ is 
strictly positive on~$\cZ_{\Delta}$, and it is zero on the 
complement~$\cZ_{\Delta}^{c}$.  Since the transfer operator $\L$ on 
the subsystem $\Sigma_{\Delta}$ coincides with $\L_{\Delta}$ for 
points in~$\Sigma_{\Delta}$, we conclude from~(ii) that 
$h_{\pdelta}|_{\Sigma_{\Delta}}\equiv w_{\Delta}$.

Although not explicitly mentioned in~\cite{CMS}, the function 
$h_{\pdelta}$ is a H\"older continuous function with the same H\"older 
exponent of the potential $\phi$.  This is because from~(i) and~(ii) 
we have $h_{\pdelta}=\lim_{n\to\infty}\, 
e^{-nP_{\Delta}}\,\L_{\Delta}^{n}(\1)$ in the supremum norm
$\|\cdot\|_{\infty}$ on $\Sigma_A^+$.
Hence 
$\|e^{-nP_{\Delta}}\,\L_{\Delta}^{n}(\1)\|_{\infty}$ is a 
bounded sequence.  Now, if $x,y\in\Sigma_{A}^{+}$ are such that 
$x_{i}=y_{i}$, for $i=0,\ldots,k-1$ and $k\geq 1$, then either 
$x,y\in\cZ_{\Delta}$ or $x,y\in\cZ_{\Delta}^{c}$.  Therefore we have
\[
\begin{split}
	e^{-nP_{\Delta}}\, &\big|
	\L_{\Delta}^{n}\1(x) - \L_{\Delta}^{n}\1(y) \big| 
	\; \\ %
	&\leq\; \sum_{ \stackrel{\scriptstyle (i_{0}\to\cdots\to i_{n-1}) 
	}{i_{j}\in\Delta} } e^{-nP_{\Delta}}\, e^{\phi(i_{0},\ldots, 
	i_{n-1},y)}\, \big| e^{\phi(i_{0},\ldots, 
	i_{n-1},x)-\phi(i_{0},\ldots, i_{n-1},y)} - 1 \big| \\ %
	&\leq\; \|e^{-nP_{\Delta}}\,\L_{\Delta}^{n}(\1)\|_{\infty} 
	\, \big| e^{c\,\theta^{n+k}} - 1 \big| \;\leq\; C\,\theta^{k}\;,
\end{split}
\]
where $C$ is independent of $n,k$ and $x,y$.  Hence 
$\mbox{var}_{k}(h_{\pdelta})\leq C\,\theta^{k}$ and $h_{\pdelta}$ is 
$\theta$-H\"older.

From~(ii) one can extend the convergence from 
$C_{p}^{+}(\Sigma_{A}^{+})$ to 
$C(\Sigma_{A}^{+})$, which is the same proof as 
in~\cite{CMS}.  This proves Proposition~\ref{prop:cms} as stated in 
the present paper.  For the remaining comments in 
Section~\ref{sec:sft}, we mention the corresponding changes in the 
expressions (9), (10) and (11) of the main result of~\cite{CMS} in 
our setting. Consider the Pianigiani-Yorke measure $\mu_{PY}$ defined 
by
\[
	\mu_{PY}(B) \;=\; \int_{B}\,h_{\pdelta}\,d\mu \;,
\]
for every Borel subset $B\subseteq\Sigma_{A}^{+}$.  First we note that 
for $f,g\in L^{1}(\mu)$ we have for every $n\geq 1$, 
$\L_{\Delta}^{n}(f\cdot\bigchi_{\Delta_{n}}) = \L_{\Delta}^{n}(f)$ and 
\[
	\L_{\Delta}^{n}(f\cdot g\scirc \S^{n}) \;=\; 
	\L_{\Delta}^{n}(\bigchi_{\Delta_{n}}\cdot f\cdot g\scirc \S^{n}) 
	\;=\; g\, \L_{\Delta}^{n}(f\cdot\bigchi_{\Delta_{n}}) \;=\; g \, 
	\L_{\Delta}^{n}(f) \;.
\]
On the other hand, for $n\geq 1$ we also have $ \L_{\Delta}^{n}(f\cdot 
g\scirc \S^{n}) = \L^{n}(\bigchi_{\Delta_{n}}\cdot f\cdot g\scirc 
\S^{n}) $.  Since $\mu$ is fixed by the dual operator of $\L$ we have
\begin{equation}
	\int_{\Delta_{n}} f\cdot g\scirc \S^{n}\, d\mu\;=\; \int 
	\L^{n}(\bigchi_{\Delta_{n}}\cdot f\cdot g\scirc \S^{n})\, d\mu 
	\;=\; \int g \cdot \L_{\Delta}^{n}(f) \,d\mu \;.
	\label{eq:ldelta}
\end{equation}

Let $B\subseteq\Sigma_{A}^{+}$ be a Borel subset.  Putting 
$g=\bigchi_{B}$, $f=h_{\Delta}$ in the above expression and noting 
that $\L_{\Delta}^{n}(h_{\pdelta})=e^{nP_{\pdelta}}\,h_{\Delta}$ we obtain
\[
\begin{split}
	\mu_{PY}(\S^{-n}B\cap \Delta_{n}) \;&=\; \int_{\S^{-n}B\cap 
	\Delta_{n}}\, h_{\pdelta}\,d\mu \;=\; \int_{\Delta_{n}} 
	h_{\pdelta}\cdot\bigchi_{B}\scirc \S^{n}\, d\mu \;=\\
	\int \bigchi_{B}\cdot \L_{\Delta}^{n}(h_{\pdelta}) \,d\mu %
	\;&=\; e^{nP_{\Delta}}\int \bigchi_{B}\cdot h_{\pdelta}\,d\mu \;=\; 
	e^{nP_{\Delta}}\, \mu_{PY}(B) \;.
\end{split}
\]
This proves~(\ref{eq:mupy}), since for $n=1$ we obtain
\[
	\mu_{PY}(\S^{-1}B\cap \Delta) \;=\; 
	e^{P_{\Delta}}\,\mu_{PY}(B) \;,
\]
where we identified $\Delta$ with the set $\Delta_{1}$.  Putting 
$g=\bigchi_{B}$ and $f=1$ in~(\ref{eq:ldelta}) gives
\[
	\dfrac{\mu(\S^{-n}B\cap \Delta_{n})}{\mu(\Delta_{n})} \;=
	\]
	\[
	\dfrac{\int_{\Delta_{n}} \bigchi_{B}\scirc \S^{n}\,d\mu 
	}{\int_{\Delta_{n}} d\mu } \;=\; \dfrac{\int \bigchi_{B}\cdot 
	\L_{\Delta}^{n}(\1)\, d\mu }{\int \L_{\Delta}^{n}(\1) \,d\mu } 
	\;=\; \dfrac{\int_{B} e^{-nP_{\Delta}}\,\L_{\Delta}^{n}(\1) \,d\mu 
	} {\int e^{-nP_{\Delta}}\,\L_{\Delta}^{n}(\1) \,d\mu } \;.
\]
Taking the limit when $n\to\infty$ proves~(\ref{eq:cond}), since
\[
	\lim_{n\to\infty}\, \mu(\S^{-n}B \,|\, \Delta_{n}) \;=\; 
	\dfrac{\int_{B} h_{\Delta}\,d\mu } {\int h_{\Delta}\,d\mu } \;=\; 
	\dfrac{\mu_{PY}(B)}{\mu_{PY}(\Sigma_{A}^{+}) } \;.
\]
Although $\Sigma_{\Delta}$ (which is the support of $\mu_{\Delta}$) 
has $\mu$-measure zero, an interesting fact is that 
\begin{equation}
	\mu_{\Delta}(B) \;=\; \lim_{n\to\infty} \, \mu(B\,|\,\Delta_{n}) 
	\;,
	\label{eq:db}
\end{equation}
for every closed and open subset $B\subseteq\Sigma_{A}^{+}$.  (Since 
$\mu$ and $\mu_{\pdelta}$ are ergodic measures for $\S$, they are 
mutually singular, therefore the above is untrue in general for all 
Borel sets $B$.)  Now, assume $B$ is a closed and open subset of 
$\Sigma_{A}^{+}$ and then $g=\bigchi_{B}$ is a continuous function on 
$\Sigma_{A}^{+}$.  Note that
\[
	\mu(B\,|\,\Delta_{n}) \;=\; 
	\dfrac{\mu(B\cap\Delta_{n})}{\mu(\Delta_{n})} \;=\; 
	\dfrac{\int_{\Delta_{n}} \bigchi_{B}\, d\mu}{\int_{\Delta_{n}} 
	d\mu} \;=\; \dfrac{\int 
	e^{-nP_{\Delta}}\,\L_{\Delta}^{n}(\bigchi_{B})\, 
	d\mu}{\int e^{-nP_{\Delta}}\,\L_{\Delta}^{n}(\1)\, d\mu} \;.
\]
Taking the limit when $n\to\infty$ and using an extension of 
Lemma~\ref{lem:cms3} (ii) to continuous functions, we 
obtain~(\ref{eq:db}).


\bibliographystyle{article}

\end{document}